\documentclass[10pt,a4paper]{amsart}
\usepackage{amssymb,amsmath}
\usepackage[cm,headings]{fullpage}
\usepackage{array}
\usepackage{booktabs}
\usepackage{multirow}
\usepackage{url}
\usepackage{cellspace}
\usepackage{makecell}
\usepackage{multirow, array}

\DeclareMathOperator{\mydim}{dim_2}

\newtheorem{theorem}{Theorem}
\newtheorem{corollary}[theorem]{Corollary}

\newtheorem{proposition}[theorem]{Proposition}

\theoremstyle{definition}

\theoremstyle{definition}

\theoremstyle{definition}

\theoremstyle{definition}

\title{The two-distance sets in dimension four}
\author{Ferenc Sz\"oll\H{o}si}
\thanks{\today, preprint. This research was supported in part by the Academy of Finland, Grant \#289002}
\address{F. Sz.: Department of Communications and Networking, Aalto University School of Electrical Engineering, P.O. Box 15400, 00076 Aalto, Finland}
\email{szoferi@gmail.com}

\begin{document}
\begin{abstract}
A finite set of distinct vectors $\mathcal{X}$ in the $d$-dimensional Euclidean space $\mathbb{R}^d$ is called a $2$-distance set, if the set of mutual distances between distinct elements of $\mathcal{X}$ has cardinality exactly $2$. In this note we classify the $2$-distance sets in $\mathbb{R}^4$ up to isometry with computer-aided methods.
\end{abstract}
\maketitle
\section{Introduction and main results}
Let $d\geq1$ be an integer, and let $\mathbb{R}^d$ denote the $d$-dimensional Euclidean space equipped with the standard inner product $\left\langle.,.\right\rangle$ and norm induced metric $\mu$. Following the terminology of \cite{M}, a Euclidean representation of a simple graph $\Gamma$ on $n\geq1$ vertices is an embedding $f$ (with real parameters $\alpha_2>\alpha_1>0$) of the vertex set of $\Gamma$ into $\mathbb{R}^d$ such that for different vertices $u\neq v$ we have $\mu(f(u),f(v))=\alpha_1$ if and only if $\{u,v\}$ is an edge of $G$, and $\mu(f(u),f(v))=\alpha_2$ otherwise. The smallest $d$ for which such a representation exist is denoted by $\mydim\Gamma$. If $\Gamma$ is neither complete, nor empty, then its image under $f$ is called an $n$-element $2$-distance set \cite{BBS}, \cite{ES}, \cite{EF}, \cite{L}, \cite{N}, \cite{WX}. The representation, as well as the $2$-distance set is called spherical, if the image of $f$ lies on the $(d-1)$-sphere of radius $1$ in $\mathbb{R}^d$ \cite{M2}, \cite{M3}. A spherical representation is called J-spherical \cite[Definition~4.1]{M}, if $\alpha_1=\sqrt2$. Graphs on $n\geq d+2$ vertices having a $J$-spherical representation in $\mathbb{R}^d$ are in a certain sense extremal \cite{RK}. We remark that other authors relax the condition $\alpha_2>\alpha_1$ thus essentially identifying the same $2$-distance set with a graph $\Gamma$ and its complement $\overline{\Gamma}$ \cite{NS}, \cite{R}.

Motivated by a recent problem posed in \cite[Section~4.3]{M}, we continue the computer-aided generation and classification of $2$-distance sets in Euclidean spaces \cite{FSZ}, a program initiated originally in \cite{L}. In particular, we describe the $2$-distance sets in $\mathbb{R}^4$, that is, we determine all simple graphs $\Gamma$ with $\mydim \Gamma=4$. Since all such graphs are known on at most $6$ vertices \cite{ES}, \cite{M}, and it is known that there are no such graphs on more than $10$ vertices \cite{L}, the aim of this note is to close this gap by classifying the graphs in the remaining cases. The main result is the following.
\begin{theorem}\label{mytheorem1}
The number of $n$-element $2$-distance sets in $\mathbb{R}^4$ for $n\in\{7,8,9\}$ is $33$, $20$, and $5$ up to isometry.
\end{theorem}
The proof is in part computational, and easily follows from the theory developed earlier in \cite{FSZ}, which we briefly outline here for completeness as follows. Assume that $\Gamma$ is a graph with vertices $v_1$, $\dots$, $v_n$. Let $a$ and $b$ be indeterminates, and associate to $\Gamma$ a ``candidate Gram matrix'' $G(a,b):=aA(\Gamma)+bA(\overline{\Gamma})+I$, where $A(\Gamma)$ is the graph adjacency matrix, and $I$ is the identity matrix of order $n$. Now let $f$ be a spherical representation of $\Gamma$ (with parameters $\alpha_2>\alpha_1>0$ as usual) in $\mathbb{R}^d$. Then the Gram matrix of the representation can be written as $[\left\langle f(v_i),f(v_j)\right\rangle]_{i,j=1}^n=G(1-\alpha_1^2/2,1-\alpha_2^2/2)$. This correspondence allows us to construct a representation based on solely $A(\Gamma)$ by exploiting that the rank of $G(1-\alpha_1^2/2,1-\alpha_2^2/2)$ is at most $d$. Indeed, if we are given a candidate Gram matrix $G(a,b)$, then those values $a^\ast,b^\ast\in\mathbb{C}$ for which $G(a^\ast,b^\ast)$ has a certain rank can be found by considering the set of $(d+1)\times (d+1)$ minors of $G(a,b)$ which should all be vanishing. The arising system of polynomial equations can be analyzed by a standard Gr\"obner basis computation \cite{cocoa}, \cite{GBBOOK}, as detailed in \cite{FSZ}. In particular, if no common solutions are found, then the candidate Gram matrix (as well as both $\Gamma$ and its complement) should be discarded as it cannot correspond to a spherical $2$-distance set in $\mathbb{R}^d$. On the other hand, if some solutions are found, then the candidate Gram matrix survives the test, and one should further ascertain that $G(a^\ast,b^\ast)$ is a positive semidefinite matrix. This can be done by investigating the signs of the coefficients of its characteristic polynomial \cite[Corollary~7.2.4]{HJ}.

The general case (i.e., when $f$ is not necessarily spherical) is analogous, but slightly more technical as the image of $f$ should be translated to the origin first, and then the Gram matrix of this shifted set (which is sometimes called Menger's matrix) should be considered \cite[Section~7.1]{L}, \cite[Section~4]{FSZ}. In particular, we have
\begin{equation*}
\left\langle f(v_i)-f(v_n),f(v_j)-f(v_n)\right\rangle=\left(G(\alpha_1^2,\alpha_2^2)_{i,n}+G(\alpha_1^2,\alpha_2^2)_{j,n}-G(\alpha_1^2,\alpha_2^2)_{i,j}+I_{ij}\right)/2,\qquad i,j\in\{1,\dots,n-1\}.
\end{equation*}
The right hand side describes the entries of a positive semidefinite matrix of rank at most $d$, which depends on $A(\Gamma)$ only. This rank condition can be treated in a similar way as discussed previously. In Table~\ref{Table1} we summarize the number of surviving candidate Gram matrices found by a simple backtrack search, and the number of corresponding $2$-distance sets. The entry marked by an asterisk indicates that $6$ out of the $42$ cases are actually the maximum $2$-distance sets in $\mathbb{R}^3$, see \cite[Section~10]{ES}. The proof of Theorem~\ref{mytheorem1} can be obtained by setting $d=4$ and then analyzing one by one the surviving candidate Gram matrices and the corresponding graphs on $n\in\{7,8,9\}$ vertices. 

\begin{table}[htbp]%
\tiny
\begin{tabular}{lrrrrrr}
\hline
$n$                                   &   6 &   7 &    8 &      9 &  10 & 11\\
\hline
\#Graphs up to complements            &  78 & 522 & 6178 & 137352 & 6002584 & 509498932\\
\#Surviving candidate Gram matrices (all)       &  77 &  22 &   13 &      4 &       1 &  0\\
\#Surviving candidate Gram matrices (spherical) &  30 &  17 &    6 &      2 &       1 &  0\\
\#Spherical $2$-distance sets         &  $^\ast$42 &  23 &    7 &      2 &       1 &  0\\
\#Nonspherical $2$-distance sets      & 103 &  10 &   13 &      3 &       0 &  0\\
\hline
\end{tabular}
\caption{The number of candidate Gram matrices and $2$-distance sets in $\mathbb{R}^4$}
\label{Table1}
\normalsize
\end{table}
It is known that the maximum cardinality of a $2$-distance set in $\mathbb{R}^4$ is exactly $10$, and the unique configuration realizing this corresponds to the triangular graph $T(5)$, see \cite{L}. We have verified this result independently. Indeed, our computer program identified a single $10\times 10$ candidate Gram matrix, which cannot be extended any further, and whose spherical representation is shown in Table~\ref{TableSph8910}. While the subgraphs of $T(5)$ are obviously spherical $2$-distance sets embedded in $\mathbb{R}^4$, there are several additional examples as follows.
\begin{proposition}
The number of $9$-point $2$-distance sets in $\mathbb{R}^4$ is $5$, out of which $2$ are spherical.
\end{proposition}
\begin{proof}
Our computer program generated $4$ candidate Gram matrices in the general case, and $2$ in the spherical case, see Table~\ref{Table1}. The two spherical cases correspond to the $9$-vertex subgraph of $T(5)$, and to the Paley-graph, see Table~\ref{TableSph8910}. The remaining two candidate Gram matrices correspond to three nonspherical $2$-distance sets, see Table~\ref{TableGen789}.
\end{proof}
\begin{proposition}
The number of $8$-point $2$-distance sets in $\mathbb{R}^4$ is $20$, out of which $7$ are spherical.
\end{proposition}
\begin{proof}
Our computer program generated $13$ candidate Gram matrices in the general case, and $6$ in the spherical case, see Table~\ref{Table1}. In the spherical case five out of the six candidate Gram matrices yielded one spherical $2$-distance set, while one resulted in two, see Table~\ref{TableSph8910}. The remaining $7$ candidate Gram matrices correspond to two nonspherical $2$-distance sets each, except for the one which is self-complementary. See Table~\ref{TableGen789} for some details.
\end{proof}
\begin{table}[htbp]%
\tiny
\begin{tabular}{rllll}
\hline
$n$  & $G(a,b)$ & $G(1,0)$ & $(a^\ast,b^\ast)$ & Remark\\
\hline
$10$ & \texttt{aaaaaaaabbababaabbaaabaabaabbabaabaabbaabaaaa} & $\Gamma_{10A}$ & $(1/6,-2/3)$ & $T(5)$, $\mydim(\overline{\Gamma}_{10A})=5$\\
$9$  & \texttt{aaaaaaaabbababaabbaaabaabaabbabaabaa} & $\Gamma_{9A}$ & $(1/6,-2/3)$ & $\Gamma_{9A}\sim\Gamma_{10A}\setminus\{\ast\}$, $\mydim(\overline{\Gamma}_{9A})=5$\\
$9$  & \texttt{aaaabbabbabababbabbaabbaababbbababaa} & $\Gamma_{9B}$ & $(1/4,-1/2)$ & Paley, self-complementary\\
$8$  & \texttt{aaaaaaaaabaabaaabaaaabaaaaaa} & $\Gamma_{8A}$ & $(0,-1)$ & $16$-cell, J-spherical, $\mydim(\overline{\Gamma}_{8A})=7$\\
$8$  & \texttt{aaaaaaaaabaabababaabbbaaabbb} & $\Gamma_{8B}$ & $((1\pm\sqrt5)/4,1/2)$ & \\
$8$  & \texttt{aaaaaaaabbababababaabbbaabaa} & $\Gamma_{8C}$ & $(1/6,-2/3)$ & $\mydim(\overline{\Gamma}_{8C})=5$\\
$8$  & \texttt{aaaaaaabbbbabbabbabaabbbaaaa} & $\Gamma_{8D}$ & $(1/5,-3/5)$ & $\mydim(\overline{\Gamma}_{8D})=6$\\
$8$  & \texttt{aaaaaaaabbababaabbaaabaabaab} & $\Gamma_{8E}$ & $(1/6,-2/3)$ & $\Gamma_{8E}\sim\Gamma_{9A}\setminus\{\ast\}$, $\mydim(\overline{\Gamma}_{8E})=5$\\
$8$  & \texttt{aaaabbabbabababbabbaabbaabab} & $\Gamma_{8F}$ & $(1/4,-1/2)$ & $\Gamma_{8F}\sim\Gamma_{9B}\setminus\{\ast\}$, self-complementary\\
\hline
\end{tabular}
\caption{Spherical $2$-distance sets on $n\in\{8,9,10\}$ points in $\mathbb{R}^4$}
\label{TableSph8910}
\normalsize
\end{table}
In the tables the vectorization (i.e., row-wise concatenation) of the lower triangular part of a graph adjacency matrix of order $n$ is denoted by a string of letters \texttt{a} and \texttt{b} of length $n(n-1)/2$, where letter \texttt{a} indicates adjacent vertices.
\begin{proposition}
The number of $7$-point $2$-distance sets in $\mathbb{R}^4$ is $33$, out of which $23$ are spherical.
\end{proposition}
\begin{proof}
Our computer program generated $22$ candidate Gram matrices in the general case, and $17$ in the spherical case, see Table~\ref{Table1}. In the spherical case there was a single matrix which did not correspond to any $2$-distance sets as it turned out to be indefinite. All the other candidate Gram matrices yielded at least one spherical $2$-distance set, see Table~\ref{TableSph7}. The remaining $5$ candidate Gram matrices correspond to two nonspherical $2$-distance sets each, see Table~\ref{TableGen789}.
\end{proof}
\begin{table}[htbp]%
\tiny
\begin{tabular}{rllll}
\hline
$n$  & $G(a,b)$ & $G(1,0)$ & $(a^\ast,b^\ast)$ & Remark\\
\hline
$7$ & \texttt{aaaaaaaaabaabaaabaaaa} & $\Gamma_{7A}$ & $(0,-1)$ & J-spherical, $\mydim(\overline{\Gamma}_{7A})=6$\\
$7$ & \texttt{aaaaaaaaabaababaabbaa} & $\Gamma_{7B}$ & $((1\pm\sqrt{5})/4,1/2)$ &\\
$7$ & \texttt{aaaaaaaaabaabababaabb} & $\Gamma_{7C}$ & $((1\pm\sqrt{5})/4,1/2)$ &\\
$7$ & \texttt{aaaaaaaabbababaabbaaa} & $\Gamma_{7D}$ & $(1/6,-2/3)$ & $\mydim(\overline{\Gamma}_{7D})=5$\\
$7$ & \texttt{aaaaaaaabbababababaab} & $\Gamma_{7E}$ & $(1/6,-2/3)$ & $\mydim(\overline{\Gamma}_{7E})=5$\\
$7$ & \texttt{aaaaaaaabbababbabbabb} & $\Gamma_{7F}$ & $((1\pm\sqrt5)/4,1/2)$ &\\
$7$ & \texttt{aaaaaaaabbababbbaabbb} & $\Gamma_{7G}$ & $((1\pm\sqrt5)/4,1/2)$ &\\
$7$ & \texttt{aaaaaaabbbbabbabbabaa} & $\Gamma_{7H}$ & $(1/5,-3/5)$ & $\mydim(\overline{\Gamma}_{7F})=5$\\
$7$ & \texttt{aaaaaaaabbabababaabaa} & $\Gamma_{7I}$ & $(1/6,-2/3)$ & $\mydim(\overline{\Gamma}_{7G})=5$\\
$7$ & \texttt{aaaaababaabaaaabbbbbb} & $\Gamma_{7J}$ & $(1/4,-1/2)$ & $\mydim(\overline{\Gamma}_{7H})=5$\\
$7$ & \texttt{aaaaababaabaaabbbaaaa} & $\Gamma_{7K}$ & $((-1-\sqrt5)/8,3(-1+\sqrt5)/8)$ & $\mydim(\Gamma_{7K})=5$\\
$7$ & \texttt{aaaaababaabaabbbbbaaa} & $\Gamma_{7L}$ & $(1/6,-2/3)$ & $\mydim(\overline{\Gamma}_{7L})=5$\\
$7$ & \texttt{aaaaababababbaaabbbbb} & $\Gamma_{7M}$ & $((1\pm\sqrt5)/4,1/2)$ &\\
$7$ & \texttt{aaaaababababbababbbaa} & $\Gamma_{7N}$ & $(-5/12,7/24)$ & $\mydim(\Gamma_{7N})=5$\\
$7$ & \texttt{aaaaabababbabaabbaaaa} & $\Gamma_{7O}$ & $(\alpha^{\pm},-1/2-2\alpha^{\pm})$ & $8\alpha^3+32\alpha^2+10\alpha-1=0, |\alpha^\pm|\leq1$\\
$7$ & \texttt{aaaabbabbabababbabbaa} & $\Gamma_{7P}$ & $((-1\pm3)/8,(-1\mp3)/8)$ & \\
\hline
\end{tabular}
\caption{Spherical $2$-distance sets on $n=7$ points in $\mathbb{R}^4$}
\label{TableSph7}
\normalsize
\end{table}
\begin{proposition}[\mbox{cf.~\cite[p.~494]{ES}, \cite[Section~4.3]{M}}]
The number of $6$-point $2$-distance sets in $\mathbb{R}^4$ is $145$. The number of $6$-point spherical $2$-distance sets in $\mathbb{R}^4$ is $42$, out of which $6$ are in fact the maximum $2$-distance sets in $\mathbb{R}^3$.
\end{proposition}
\begin{proof}
It is known, see \cite{ES}, \cite{M}, that a graph on $6$ vertices can be represented in $\mathbb{R}^4$ unless it is a disjoint union of cliques. Since the total number of simple graphs on $6$ vertices is $156$, out of which $11$ are disjoint union of cliques, we find that $145$ graphs can be represented in $\mathbb{R}^4$. Our computer program generated $30$ candidate Gram matrices in the spherical case, see Table~\ref{Table1}. There were two indefinite matrices, and the remaining $28$ resulted in at least one spherical $2$-distance set each. Amongst these, we found the $6$ maximum $2$-distance sets in $\mathbb{R}^3$, denoted by $\Gamma_{6K}$, $\Gamma_{6O}$, $\overline{\Gamma}_{6O}$, $\Gamma_{6R}$, $\overline{\Gamma}_{6R}$, and $\Gamma_{6Y}$, see Table~\ref{TableSph6}.
\end{proof}
Finally, there are $7$ graphs $\Gamma$ on $n=5$ vertices for which $\mydim\Gamma=4$. One particular spherical representation is given of these in Table~\ref{TableSph5}. The number of corresponding nonisometric $2$-distance sets in $\mathbb{R}^4$ in these cases is infinite.
\begin{table}[htbp]%
\tiny
\begin{tabular}{rllll}
\hline
$n$  & $G(a,b)$ & $G(1,0)$ & $(a^\ast,b^\ast)$ & Remark\\
\hline
$6$ & \texttt{aaaaaaaaabaabaa} & $\Gamma_{6A}$ & $(0,-1)$ & J-spherical, $\mydim(\overline{\Gamma}_{6A})=5$\\
$6$ & \texttt{aaaaaaaaabaabab} & $\Gamma_{6B}$ & $((1\pm\sqrt5)/4,1/2)$ &\\
$6$ & \texttt{aaaaaaaabbaabba} & $\Gamma_{6C}$ & $(-1/3,1/3)$ & $\mydim(\Gamma_{6C})=4$\\
$6$ & \texttt{aaaaaaaabbababa} & $\Gamma_{6D}$ & $(1/6,-2/3)$ & $\mydim(\overline{\Gamma}_{6D})=4$\\
$6$ & \texttt{aaaaaaaabbababb} & $\Gamma_{6E}$ & $((1\pm\sqrt5)/4,1/2)$ &\\
$6$ & \texttt{aaaaaaaabbbbaaa} & $\Gamma_{6F}$ & $(0,\pm\sqrt{2}/2)$ & $\Gamma_{6F}$ is J-spherical\\
$6$ & \texttt{aaaaaaaabbbbaab} & $\Gamma_{6G}$ & $((-1\pm3)/12,(-1\mp3)/6)$ &\\
$6$ & \texttt{aaaaaaabbbbabba} & $\Gamma_{6H}$ & $(1/5,-3/5)$ & $\mydim(\overline{\Gamma}_{6H})=4$\\
$6$ & \texttt{aaaaaaabbbbabbb} & $\Gamma_{6I}$ & $(1/2,(1\pm\sqrt5)/4)$ &\\
$6$ & \texttt{aaaaababaaabbbb} & $\Gamma_{6J}$ & $(1/3,-1/3)$ & $\mydim(\overline{\Gamma}_{6J})=4$\\
$6$ & \texttt{aaaaababaabaaaa} & $\Gamma_{6K}$ & $(0,-1)$ & octahedron, J-spherical, $\mydim(\overline{\Gamma}_{6K})=5$\\
$6$ & \texttt{aaaaababaabaaab} & $\Gamma_{6L}$ & $((-1-\sqrt5)/8,3(-1+\sqrt5)/8)$ & $\mydim(\Gamma_{6L})=4$\\
$6$ & \texttt{aaaaababaabaabb} & $\Gamma_{6M}$ & $(1/6,-2/3)$ & $\mydim(\overline{\Gamma}_{6M})=4$\\
$6$ & \texttt{aaaaababaabbbbb} & $\Gamma_{6N}$ & $(1/4,-1/2)$ & $\mydim(\overline{\Gamma}_{6N})=4$\\
$6$ & \texttt{aaaaababababbaa} & $\Gamma_{6O}$ & $(\pm1/\sqrt5,\mp1/\sqrt5)$ & pentagonal pyramids \cite[Figure~4]{ES}\\
$6$ & \texttt{aaaaababababbab} & $\Gamma_{6P}$ & $(-5/12,7/24)$ & $\mydim(\Gamma_{6P})=4$\\
$6$ & \texttt{aaaaababababbbb} & $\Gamma_{6Q}$ & $((1\pm\sqrt5)/4,1/2)$ &\\
$6$ & \texttt{aaaaabababbaabb} & $\Gamma_{6R}$ & $(\pm1/\sqrt5,\mp1/\sqrt5)$ & pentahedrons \cite[Figure~4]{ES}\\
$6$ & \texttt{aaaaabababbabaa} & $\Gamma_{6S}$ & $(\alpha^{\pm},-1/2-2\alpha^{\pm})$ & $8\alpha^3+32\alpha^2+10\alpha-1=0, |\alpha^\pm|\leq1$\\
$6$ & \texttt{aaaaabababbabab} & $\Gamma_{6T}$ & $(\beta^{\pm},-1/2-\beta^{\pm}+2(\beta^{\pm})^2)$ & $8\beta^3-8\beta^2-2\beta+1=0, |\beta^\pm|\leq 1$\\
$6$ & \texttt{aaaaabababbbbaa} & $\Gamma_{6U}$ & $(1/6,-2/3)$ & $\mydim(\overline{\Gamma}_{6U})=4$\\
$6$ & \texttt{aaaaababbbbabba} & $\Gamma_{6V}$ & $(1/3,-1/3)$ & $\mydim(\overline{\Gamma}_{6V})=4$\\
$6$ & \texttt{aaaaabbbaabbaaa} & $\Gamma_{6W}$ & $(-3/7,2/7)$ & $\mydim(\Gamma_{6W})=4$\\
$6$ & \texttt{aaaabbabbababab} & $\Gamma_{6X}$ & $((-1\pm3)/8,(-1\mp3)/8)$ &\\
$6$ & \texttt{aaaabbbababbaaa} & $\Gamma_{6Y}$ & $((-5\pm9)/28,(-13\mp27)/56)$ & $\Gamma_{6Y}$ is a square-faced triangular prism \cite[Figure~2]{ES}\\
$6$ & \texttt{aaaabbbababbbaa} & $\Gamma_{6Z}$ & $(1/5,-3/5)$, $(-\sqrt2/3,(\sqrt2-1)/3)$ &\\
$6$ & \texttt{aaaabbbbbabbbaa} & $\Gamma\textsubscript{6{\it\AE}}$ & $((-1\pm3\sqrt3)/13,(-7\mp5\sqrt3)/26)$ &\\
$6$ & \texttt{aaabbbbbbabbbaa} & $\Gamma\textsubscript{6{\it\OE}}$ & $(-1/2,0)$ & J-spherical, $\mydim(\Gamma\textsubscript{6{\it\OE}})=5$\\
\hline
\end{tabular}
\caption{Spherical $2$-distance sets on $n=6$ points in $\mathbb{R}^4$}
\label{TableSph6}
\normalsize
\end{table}
\begin{table}[htbp]%
\tiny
\begin{tabular}{rllll}
\hline
$n$  & $G(a,b)$ & $G(1,0)$ & $(a^\ast,b^\ast)$ & Remark\\
\hline
$5$ & \texttt{aaaaaaaaaa} & $\Gamma_{5A}$ & $a^\ast=-1/4$ & regular $5$-cell, $\mydim(\overline{\Gamma}_{5A})=4$, $1$-distance set\\
$5$ & \texttt{aaaaaaaaab} & $\Gamma_{5B}$ & $((1-\sqrt7)/6,0)$ & J-spherical, $\mydim(\Gamma_{5B})=3$\\
$5$ & \texttt{aaaaaabbbb} & $\Gamma_{5C}$ & $(0,-1/2)$ & J-spherical, $\mydim(\overline{\Gamma}_{5C})=3$\\
$5$ & \texttt{aaaaabaabb} & $\Gamma_{5D}$ & $(-1/3,0)$ & J-spherical, $\mydim(\Gamma_{5D})=3$\\
$5$ & \texttt{aaaaababaa} & $\Gamma_{5E}$ & $((1-\sqrt5)/4,0)$ & J-spherical, $\mydim(\Gamma_{5E})=3$\\
$5$ & \texttt{aaabbbbbba} & $\Gamma_{5F}$ & $(0,-1/\sqrt6)$ & J-spherical, $\mydim(\overline{\Gamma}_{5F})=3$\\
\hline
\end{tabular}
\caption{Spherical $2$-distance sets on $n=5$ points in $\mathbb{R}^4$}
\label{TableSph5}
\normalsize
\end{table}
\begin{table}[htbp]%
\tiny
\begin{tabular}{rllll}
\hline
$n$  & $G(a,b)$ & $G(1,0)$ & $(a^\ast,b^\ast)$ & Remark\\
\hline
$9$ & \texttt{aaaaaaaaaaaabbbababbbabbabbbabbbabbb} & $\Gamma_{9C}$ & $(1,(3\pm\sqrt5)/2)$ &\\
$9$ & \texttt{aaaaaaaaabaababaabbaaabbbbbbbabbbbbb} & $\Gamma_{9D}$ & $(1,(3+\sqrt5)/2)$ & self-complementary\\
$8$ & \texttt{aaaaaaaaaaaabbbababbbabbabbb} & $\Gamma_{8G}$ & $(1,(3\pm\sqrt5)/2)$ & $\Gamma_{8G}\sim\Gamma_{9C}\setminus\{\ast\}$\\
$8$ & \texttt{aaaaaaaaaaaabbbababbbbaabbbb} & $\Gamma_{8H}$ & $(1,(3\pm\sqrt5)/2)$ & \\
$8$ & \texttt{aaaaaaaaabaababaabbaaabbbbbb} & $\Gamma_{8I}$ & $(1,(3\pm\sqrt5)/2)$ & $\Gamma_{8I}\sim\Gamma_{9D}\setminus\{\ast\}$\\
$8$ & \texttt{aaaaaaaaabaababaabbaabbbbbbb} & $\Gamma_{8J}$ & $(1,(3\pm\sqrt5)/2)$ &  \\
$8$ & \texttt{aaaaaaaaabaabababaabbabbbbbb} & $\Gamma_{8K}$ & $(1,(3\pm\sqrt5)/2)$ & \\
$8$ & \texttt{aaaaaaaaabaabababaabbbbbbbbb} & $\Gamma_{8L}$ & $(1,(3\pm\sqrt5)/2)$ & $\Gamma_{8L}\sim\overline{\Gamma}_{9C}\setminus\{\ast\}$\\
$8$ & \texttt{aaaaaaaaabaabababbbbbbabbbbb} & $\Gamma_{8M}$ & $(1,(3+\sqrt5)/2)$ & $\Gamma_{8M}\sim\Gamma_{9D}\setminus\{\ast\}$, self-complementary\\
$7$ & \texttt{aaaaaaaaaaaabbbababbb} & $\Gamma_{7Q}$ & $(1,(3\pm\sqrt5)/2)$ & $\Gamma_{7Q}\sim\Gamma_{9C}\setminus\{\ast,\ast\}$\\
$7$ & \texttt{aaaaaaaaabaabababbbbb} & $\Gamma_{7R}$ & $(1,(3\pm\sqrt5)/2)$ & $\Gamma_{7R}\sim\Gamma_{9D}\setminus\{\ast,\ast\}$\\
$7$ & \texttt{aaaaaaaaabaababbbbbbb} & $\Gamma_{7S}$ & $(1,(3\pm\sqrt5)/2)$ & $\Gamma_{7S}\sim\overline{\Gamma}_{9C}\setminus\{\ast,\ast\}$\\
$7$ & \texttt{aaaaaaaaababbbbbabbbb} & $\Gamma_{7T}$ & $(1,(3\pm\sqrt5)/2)$ & $\Gamma_{7T}\sim\Gamma_{9D}\setminus\{\ast,\ast\}$\\
$7$ & \texttt{aaaaaaaabbababbabbbbb} & $\Gamma_{7U}$ & $(1,(3\pm\sqrt5)/2)$ & $\Gamma_{7U}\sim\Gamma_{9C}\setminus\{\ast,\ast\}$\\
\hline
\end{tabular}
\caption{General (nonspherical) $2$-distance sets on $n\in\{7,8,9\}$ points in $\mathbb{R}^4$}
\label{TableGen789}
\normalsize
\end{table}
\begin{corollary}
The number of graphs $\Gamma$ for which $\mydim\Gamma=4$ is $211$.
\end{corollary}
\begin{proof}
This follows from earlier results in \cite{ES}, \cite{L}, and Theorem~\ref{mytheorem1}: the number of such graphs on $n\in\{5,6,7,8,9,10\}$ vertices is $7$, $145$, $33$, $20$, $5$, and $1$, respectively, and there are no such graphs on $n<5$ or $n>10$ vertices.
\end{proof}
We conclude this manuscript with the following remark: the classification of the maximum $3$-distance sets in $\mathbb{R}^4$ has recently been carried out in \cite{FSZ}, and therefore data on the (not necessarily largest) candidate Gram matrices is readily available for that case too (see \cite[Table~5 and 7]{FSZ}). However, the individual analysis and ultimately the presentation of those tens of thousands of matrices would require considerably more efforts.

\end{document}